\documentclass[11pt]{amsart}

\usepackage{amsfonts,amsmath,amssymb}
\usepackage{pinlabel}
\usepackage{amsthm}
\usepackage{latexsym}
\usepackage{graphicx}
\usepackage{psfrag}
\usepackage{color}

\copyrightinfo{2006}{American Mathematical Society}

\newtheorem{theorem}{Theorem}[section]
\newtheorem{lemma}[theorem]{Lemma}

\newtheorem{prop}[theorem]{Proposition}

\theoremstyle{definition}

\newtheorem{remark}[theorem]{Remark}

\newcommand{\Nbd}{\operatorname{Nbd}}
\newcommand{\cl}{\operatorname{cl}}

\numberwithin{equation}{section}

\begin{document}

\title[Genus two Goeritz groups of lens spaces]{Genus two Goeritz groups of lens spaces}

\author{Sangbum Cho}\thanks{This work is supported by Basic Science Research Program through the National Research Foundation of Korea(NRF) funded by the Ministry of Education, Science and Technology (2012006520).}
\address{Department of Mathematics Education, Hanyang University, Seoul 133-791,
Korea}
\email{scho@hanyang.ac.kr}


\date{\today}

\begin{abstract}
Given a genus-$g$ Heegaard splitting of a $3$-manifold, the Goeritz group is defined to be the group of isotopy classes of orientation-preserving homeomorphisms of the manifold that preserve the splitting.
In this work, we show that the Goeritz groups of genus-$2$ Heegaard splittings for lens spaces $L(p, 1)$ are finitely presented, and give explicit presentations of them.
\end{abstract}

\maketitle

\section{Introduction}
\label{sec:intro}

It is well known that every closed orientable $3$-manifold $M$ can be decomposed into two handlebodies of the same genus.
This is what we call a Heegaard splitting of the manifold, and the genus of the handlebodies is called the genus of the splitting.
Given a genus-$g$ Heegaard splitting of $M$, the {\it Goeritz group} of the splitting, which we will denote by  $\mathcal G_g$, is the group of isotopy classes of orientation-preserving homeomorphisms of $M$ that preserve each of the handlebodies of the splitting setwise.
In particular, this group is interesting when the manifold is the $3$-sphere or a lens space since it is well known from \cite{W}, \cite{B} and \cite{B-O}, that they have unique Heegaard splittings for each genus up to isotopy.
In this case, each Goeritz group depends only on the genus of the splitting, and so we can define the {\it genus-$g$ Goeritz group} $\mathcal G_g$ of each of those manifolds without mentioning a specific Heegaard splitting.
For the $3$-sphere, it was shown in \cite{Ge} and \cite{Sc} that $\mathcal G_2$ is finitely generated, and subsequently in \cite{Ak} and \cite{C} that $\mathcal G_2$ is finitely presented and its finite presentation was introduced.
Further, in \cite{Kod}, a natural generalization of a Goeritz group is studied, the group of isotopy classes of orientation-preserving homeomorphisms of the $3$-sphere preserving an embedded genus two handlebody which is possibly knotted.

In this work, we show that the Goeritz group $\mathcal G_2$ of each of the lens spaces $L(p, 1)$ is finitely presented.
In the main theorem, Theorem \ref{thm:presentation}, their explicit presentations are given.
For the genus-$2$ Goeritz groups of the other lens spaces and for the higher genus Georitz groups of the $3$-sphere and lens spaces, it is conjectured that they are all finitely presented but it is still known to be an open problem.

We generalize the method developed in \cite{C}.
We find a tree on which $\mathcal G_2$ for $L(p, 1)$ acts such that the quotient of the tree by the action of $\mathcal G_2$ is a single edge, and then apply the well known theory of groups acting on trees due to Bass and Serre \cite{S}.
Such a tree will be found in the barycentric subdivision of the disk complex for one of the handlebodies of the splitting.
For arbitrary lens spaces $L(p, q)$, finding such trees, if exist, is much more complicated problem than the case of $L(p, 1)$, which will be fully discussed in \cite{C-K}.

Throughout the paper, we simply denote by $\mathcal G$ the genus-$2$ Goeritz group $\mathcal G_2$ of a lens space.
We use the standard notation $L(p, q)$, $p \geq 2$, for a lens space with its basic properties found in standard textbooks. For example, we refer \cite{R}.
For a genus-$1$ Heegaard splitting of $L(p, 1)$, any oriented meridian circle of a solid torus of the splitting is identified with a $(p, 1)$-curve (or $(p, p-1)$-curve) on the boundary of the other solid torus, after a suitable choice of oriented longitude and meridian of the solid torus.
The triple $(V, W; \Sigma)$ will denote a genus-$2$ Heegaard splitting of a lens space $L = L(p, q)$.
That is, $L = V \cup W$ and $ V \cap W = \partial V =\partial W = \Sigma$, where $V$ and $W$ are handlebodies of genus two.

The disks $D$ and $E$ in a handlebody are always assumed to be properly embedded and their intersection is transverse and minimal up to isotopy.
In particular, if $D$ intersects $E$, then $D \cap E$ is a collection of pairwise disjoint arcs that are properly embedded in both $D$ and $E$.
Finally, $\Nbd(X)$ will denote a regular neighborhood of $X$ and $\cl(X)$ the closure of $X$ for a subspace $X$ of a polyhedral space, where the ambient space will always be clear from the context.

\section{Primitive elements of the free group of rank two}
\label{sec:free_group}
The fundamental group of the genus-$2$ handlebody is the free group $\mathbb Z \ast \mathbb Z$ of rank two.
We call an element of $\mathbb Z \ast \mathbb Z$ {\it primitive} if it is a member of a generating pair of $\mathbb Z \ast \mathbb Z$.
Primitive elements of $\mathbb Z \ast \mathbb Z$ have been well understood.
For example we refer \cite{O-Z} to the reader.
A key property of the primitive elements of the free group of rank two is the following, which is a direct consequence of Corollary 3.3 in \cite{O-Z}.

\begin{prop}
Fix a generating pair $\{x, y\}$ of $\mathbb Z \ast \mathbb Z$, and let $w$ be a primitive element of $\mathbb Z \ast \mathbb Z$. Then for some $\epsilon \in \{1,-1\}$ and some $n \in \mathbb Z$, some cyclically reduced form of $w$ is a product of terms each of the form $x^\epsilon y^n$ or $x^\epsilon y^{n+1}$, or else a product of terms each of the form $y^\epsilon x^n$ or $y^\epsilon x^{n+1}$.
\par
\label{prop:generator}
\end{prop}

From the proposition, the cyclically reduced forms of a primitive element are very restrictive.
For example, if $w$ is a primitive element of $\mathbb Z \ast \mathbb Z$, then  no cyclically reduced form of $w$ in terms of $x$ and $y$ can contain $x$ and $x^{-1}$ $($and $y$ and $y^{-1}$$)$ simultaneously.

A simple closed curve in the boundary of a genus-$2$ handlebody $W$ represents an element of $\pi_1(W) = \mathbb Z \ast \mathbb Z$.
We call a pair of essential disks in $W$ a {\it complete meridian system} for $W$ if the union of the two disks cuts off $W$ into a $3$-ball.
Given a complete meridian system $\{F, G\}$, assign symbols $x$ and $y$ to circles $\partial F$ and $\partial G$ respectively.
Suppose that an oriented simple closed curve $l$ on $\partial W$ that meets $\partial F \cup \partial G$ transversely and minimally.
Then $l$ determines a word in terms of $x$ and $y$ which can be read off from the the intersections of $l$ with $\partial F$ and $\partial G$ (after a choice of orientations of $\partial F$ and $\partial G$), and hence $l$ represents an element of the free group $\pi_1 (W) = \left< x, y\right>$.

In this set up, the following is a simple criterion for the primitiveness of the elements represented by such a simple closed curve.

\begin{lemma}
With a suitable choice of orientations of $\partial F$ and $\partial G$, if a word determined by the simple closed curve $l$ contains one of the sub-words $yxy^{-1}$ or $xyxy^n$ for $n \geq 3$, then any element in $\pi_1 (W)$ represented by $l$ cannot be a primitive element.
\par
\label{lem:criterion}
\end{lemma}

\begin{figure}
\labellist
\pinlabel {\small $f_+$} [B] at -1 118
\pinlabel {\small $f_-$} [B] at -1 62
\pinlabel {\small $g_+$} [B] at 135 118
\pinlabel {\small $g_-$} [B] at 135 62
\pinlabel {\small $f_+$} [B] at 378 118
\pinlabel {\small $f_-$} [B] at 378 62
\pinlabel {\small $h_+$} [B] at 515 118
\pinlabel {\small $h_-$} [B] at 515 62
\pinlabel {\small $l_+$} [B] at 55 110
\pinlabel {\small $l_-$} [B] at 35 95
\pinlabel {\small $m_+$} [B] at 37 80
\pinlabel {\small $m_-$} [B] at 58 64
\pinlabel {\small $c$} [B] at 231 115
\pinlabel {\small $x^2$} [B] at 415 91
\pinlabel {\small $z^2$} [B] at 477 91
\pinlabel {\large $F$} at 160 97
\pinlabel {\large$G$} at 348 97
\pinlabel {\large$H$} at 242 79
\pinlabel {\Large$\Sigma'$} [B] at 131 30
\pinlabel {\Large$W$} [B] at 331 30
\pinlabel {\Large$\Sigma''$} [B] at 513 30
\pinlabel {\large(a)} [B] at 63  0
\pinlabel {\large(b)} [B] at 251 0
\pinlabel {\large(c)} [B] at 442 0
\endlabellist
\begin{center}
\includegraphics[width=1\textwidth]{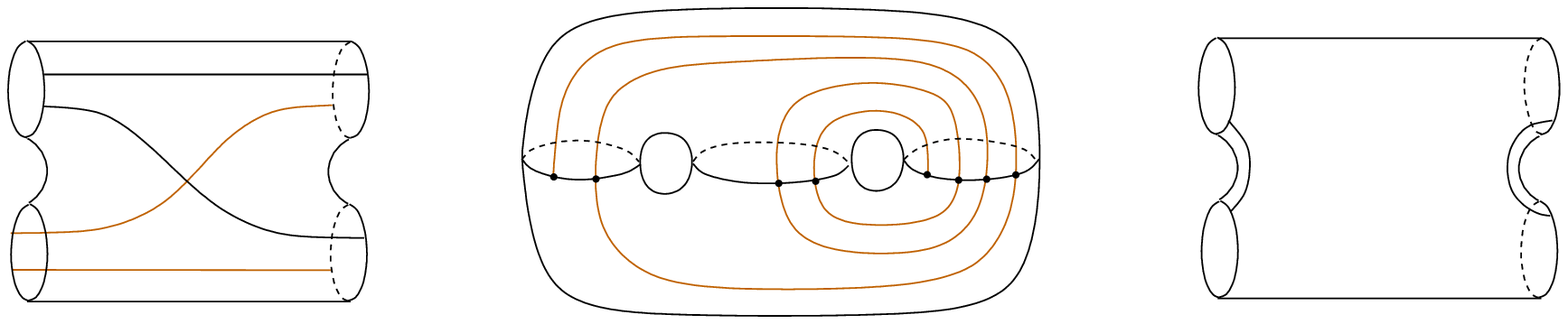}
\caption{The $4$-holed spheres $\Sigma'$ and $\Sigma''$.}
\label{nonprimitive}
\end{center}
\end{figure}

\begin{proof}
Let $\Sigma'$ be the $4$-holed sphere cut off from $\partial W$ along $\partial F \cup \partial G$, and denote by $f_+$ and $f_-$ $($respectively $g_+$ and $g_-$$)$ the boundary circles of $\Sigma'$ that came from $\partial F$ $($respectively $\partial G$$)$.

Suppose first that a word represented by $l$ contains a sub-word of the form $yxy^{-1}$.
Then we may assume that there are two arcs $l_+$ and $l_-$ of $l \cap \Sigma'$ such that $l_+$ connects $f_+$ and $g_+$, and $l_-$ connects $f_+$ and $g_-$ as in Figure \ref{nonprimitive} (a).
Since $|~l \cap f_+| = |~l \cap f_-|$ and $|~l \cap g_+| = |~l \cap g_-|$, we must have two other arcs $m_+$ and $m_-$ of $l \cap \Sigma'$ such that $m_+$ connects $f_-$ and $g_+$, and $m_-$ connects $f_-$ and $g_-$.
We see then there exists no arc component of $l \cap \Sigma'$ that meets only one of $f_+$, $f_-$, $g_+$ and $g_-$.
That is, any word determined by $l$ contains neither $x^{\pm1} x^{\mp1}$ nor $y^{\pm1} y^{\mp1}$, and so each word is cyclically reduced.
But a word determined by $l$ already contains both $y$ and $y^{-1}$, and so $l$ cannot represent a primitive element of $\pi_1(W)$ by Proposition \ref{prop:generator}.

Next, suppose that a word represented by $l$ contains a sub-word of the form $xyxy^n$ for $n \geq 3$.
We may assume there is an arc $c$ of $l \cap \Sigma'$ connecting $f_+$ and $g_+$ in $\Sigma'$.
Consider the circle which is the frontier of a regular neighborhood of $f_+ \cup c \cup g_+$ in $\Sigma'$.
This circle bounds a disk $H$ in $W$, and $\{F, H\}$ forms a complete meridian system of $W$, and so assigning symbols $x$ and $z$ to $\partial F$ and $\partial H$ respectively, the circle $l$ represents an elements of $\pi_1(W) = \langle x, z \rangle$.
See Figure \ref{nonprimitive} (b).

Let $\Sigma''$ be the $4$-holed sphere cut off from $\partial W$ along $\partial F \cup \partial H$, and denote by $f_+$ and $f_-$ $($respectively $h_+$ and $h_-$$)$ the boundary circles of $\Sigma''$ that came from $\partial F$ $($respectively $\partial H$$)$.
There are two arcs of $l \cap \Sigma''$ such that one connects $f_+$ and $f_-$, and the other one connects $h_+$ and $h_-$, and we may assume that these two arcs represent sub-word of the form $x^2$ and $z^2$ (see Figure \ref{nonprimitive} (c)).
Thus there exists no arc component of $l \cap \Sigma''$ that meets only one of $f_+$, $f_-$, $h_+$ and $h_-$.
That is, each word represented by $l$ is cyclically reduced.
But a word determined by $l$ already contains both $x^2$ and $z^2$, and so $l$ cannot represent a primitive element of $\pi_1(W)$ by Proposition \ref{prop:generator} again.
\end{proof}

\section{Primitive disks in a handlebody}
\label{sec:primitive_disks}

Recall that $(V, W; \Sigma)$ denotes a genus two Heegaard splitting of a lens space $L = L(p, q)$ with $p \geq 2$.
We call an essential disk $E$ in $V$ {\it primitive} if there exists an essential disk $E'$ in $W$ such that $\partial E$ intersects $\partial E' $ transversely in a single point.
Such a disk $E'$ is called a {\it dual disk} of $E$.
Note that $E'$ is also primitive in $W$ with a dual disk $E$, and $W \cup \Nbd(E)$ and $V \cup \Nbd(E')$ are all solid tori.
Primitive disks are necessarily non-separating.
We call a pair of disjoint, non-isotopic primitive disks in $V$ a {\it primitive pair} in $V$.
Similarly, a triple of pairwise disjoint, non-isotopic primitive disks (if exists) is a {\it primitive triple}.

A nonseparating disk $E_0$ properly embedded in $V$ is called {\it semiprimitive} if there is a primitive disk $E'$ in $W$ such that $\partial E'$ is disjoint from $\partial E_0$.
With a suitable choice of oriented meridian and longitude circles on the boundary of the solid torus $W$ cut off by $E'$, the oriented boundary circle $\partial E_0$ can be considered as a $(p, 1)$-curve on the boundary of the solid torus.

Any simple closed curve on the boundary of the solid torus $W$ represents an element of $\pi_1 (W)$ which is the free group of rank two.
We can interpret primitive disks algebraically as follows, which is a direct consequence of \cite{Go}.

\begin{lemma}
Let $D$ be a nonseparating disk in $V$.
Then $D$ is primitive if and only if $\partial D$ represents a primitive element of $\pi_1 (W)$.
\label{lem:primitive_element}
\end{lemma}

Note that no disk can be both primitive and semiprimitive since the boundary circle of a semiprimitive disk in $V$ represents the $p$-th power of a primitive element of $\pi_1(W)$.

Let $D$ and $E$ be essential disks in $V$, and suppose that $D$ intersects $E$ transversely and minimally.
Let $C \subset D$ be a disk cut off from $D$ by an outermost arc $\beta$ of $D \cap E$ in $D$ such that $C \cap E= \beta$.
We call such a $C$ an {\it outermost subdisk} of $D$ cut off by $D \cap E$.
The arc $\beta$ cuts $E$ into two disks, say $G$ and $H$.
Then we have two essential disks $E_1$ and $E_2$ in $V$ which are isotopic to disks $G \cup C$ and $H \cup C$ respectively.
We call $E_1$ and $E_2$ the {\it disks from surgery} on $E$ along the outermost subdisk $C$ of $D$ cut off by $D \cap E$.
Observe that each of $E_1$ and $E_2$ has fewer arcs of intersection with $D$ than $E$ had, since at least the arc $\beta$ no longer counts.

Since $E$ and $D$ are assumed to intersect minimally, $E_1$ (and $E_2$) is isotopic to neither $E$ nor $D$.
In particular, if both of $D$ and $E$ are nonseparating, then the resulting disks $E_1$ and $E_2$ are both nonseparating and they are not isotopic to each other, and further, $E_1$ and $E_2$ are meridian disks of the solid torus $V$ cut off by $E$, and the boundary circles $\partial E_1$ and $\partial E_2$ are not isotopic to each other in the two holed torus $\partial V$ cut off by $\partial E$.

\begin{theorem}
Let $(V, W; \Sigma)$ be the genus two Heegaard splitting of the lens space $L = L(p, 1)$, $p \geq 2$.
Let $D$ and $E$ be primitive disks in $V$ which intersect each other transversely and minimally.
Then one of the two disks from surgery on $E$ along an outermost subdisk of $D$ cut off by $D \cap E$ is primitive.
Furthermore, it has a common dual disk with $E$.
\label{thm:key}
\end{theorem}

\begin{proof}
We will prove the theorem only for $p \geq 5$.
The cases of $p \in \{2, 3, 4\}$ will be similar but simpler.

Let $C$ be an outermost subdisk of $D$ cut off by $D \cap E$.
The choice of a dual disk $E'$ of $E$ determines a unique semiprimitive disk $E_0$ in $V$.
That is, the meridian disk $E_0$ of $V$ disjoint from $E \cup E'$.
Among all the dual disks of $E$, choose one, denoted by $E'$ again, so that the semiprimitive $E_0$ determined by $E'$ intersects $C$ minimally.
Further, there is a unique semiprimitive disk $E'_0$ in $W$ disjoint from $E \cup E'$.
We give symbols $x$ and $y$ on oriented $\partial E'$ and $\partial E'_0$ respectively to have $\pi_1(W) = \langle x, y \rangle$.
For convenience, we simply identify the boundary circles $\partial E'$ and $\partial E'_0$ with the assigned symbols $x$ and $y$ respectively.
Notice that the circle $y$ is disjoint from $\partial E$ and intersects $\partial E_0$ in $p$ points in the same direction, and $x$ is disjoint from $\partial E_0$ and intersects $\partial E$ in a single point.
Thus we may assume that $\partial E_0$ and $\partial E$ determine the words $y^p$ and $x$ respectively.

Let $\Sigma_0$ be the $4$-holed sphere $\partial V$ cut off by $\partial E \cup \partial E_0$.
We regard $\Sigma_0$ as a $2$-holed annulus where the two boundary circles came from $\partial E_0$ and the two holes came from $\partial E$.
Then $y \cap \Sigma_0$ is the union of $p$ spanning arcs which cut $\Sigma_0$ into $p$ rectangles, and $x$ is a single arc connecting two holes which are contained in a single rectangle. See Figure \ref{first_sigma} (a).

\begin{figure}
\labellist
\pinlabel {\small $\alpha$} [B] at 130 149
\pinlabel {\small $\partial E_0$} [B] at 105 120
\pinlabel {\small $\partial E_0$} [B] at 7 40
\pinlabel {\small $C \cap \Sigma_0$} [B] at 104 72
\pinlabel {\small $\partial E_1$} [B] at 104 40
\pinlabel {\small $\partial E_0$} [B] at 387 120
\pinlabel {\small $\partial E_0$} [B] at 298 40
\pinlabel {\small $C \cap \Sigma_0$} [B] at 387 45
\pinlabel {\large(a)} [B] at 98 -5
\pinlabel {\large(b)} [B] at 390 -5
\endlabellist
\begin{center}
\includegraphics[width=0.7\textwidth]{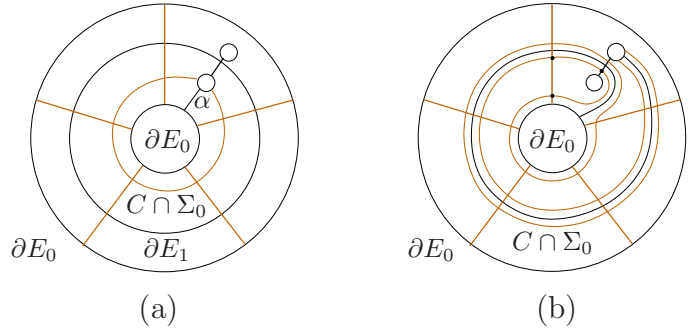}
\caption{
The $2$-holed annulus $\Sigma_0$ in $L(5, 1)$.}
\label{first_sigma}
\end{center}
\end{figure}

Suppose first that $C$ is disjoint from $E_0$.
Note that one of the disks from surgery on $E$ along $C$ is $E_0$, semiprimitive.
The arc $C \cap \Sigma_0$ is the frontier of a regular neighborhood of the union of one boundary circle of $\Sigma_0$ and an arc $\alpha$ connecting the boundary circle to a hole.
Observe that the arc $\alpha$ is disjoint from $y \cap \Sigma_0$, otherwise a word of $\partial D$ must contain $yxy^{-1}$ (after changing orientation if necessary), which contradicts that $D$ is primitive, by Lemma \ref{lem:criterion}.
See Figure \ref{first_sigma} (b).
Consequently, if we denote by $E_1$ the disk from surgery that is not $E_0$, then $\partial E_1$  intersects $\partial E'$ in a single point.
That is, the resulting disk $E_1$ is primitive with the common dual disk $E'$ of $E$.
See Figure \ref{first_sigma} (a).

From now on, we assume that $C$ intersects $E_0$.
Let $C_0$ be an outermost subdisk of $C$ cut off by $C \cap E_0$.
The arc $C_0 \cap \Sigma_0$ is the frontier of a regular neighborhood of one hole of $\Sigma_0$ and an arc, say  $\alpha_0$, connecting the hole to a boundary circle of $\Sigma_0$.
By the same reason to the case of $\alpha$, the arc $\alpha_0$ is disjoint from $y \cap \Sigma_0$.
Thus one of the disks from surgery on $E_0$ along $C_0$ is $E$, and the other one, denoted by $E_1$ again, is primitive since $\partial E_1$ intersects $\partial E'$ in a single point as in the previous case.
Note that $|C \cap E_1| < |C \cap E_0|$ from the surgery construction.
See $\Sigma_0$ in Figure \ref{sequence}.

Let $\Sigma_1$ be the $4$-holed sphere $\partial V$ cut off by $\partial E \cup \partial E_1$.
We regard $\Sigma_1$ as a $2$-holed annulus, like $\Sigma_0$, where the two boundary circles came from $\partial E_1$ and the two holes came from $\partial E$.
Then $y \cap \Sigma_1$ is the union of $p$ spanning arcs which cut $\Sigma_1$ into $p$ rectangles as in the case of $\Sigma_0$, but the two holes came from $\partial E$ are now contained in different consecutive rectangles, and also $x \cap \Sigma_1$ is the union of two arcs each joining a holes and a boundary circle of $\Sigma_1$ as in Figure \ref{sequence}.
If the original subdisk $C$ is disjoint from $E_1$, then we are done since $E_1$ is the desired primitive disk resulting from the surgery.

\begin{figure}
\labellist

\pinlabel {\small $\partial E_0$} [B] at 83 315
\pinlabel {\small $\partial E_1$} [B] at 274 315
\pinlabel {\small $\partial E_2$} [B] at 469 315
\pinlabel {\small $\partial E_0$} [B] at 5 255
\pinlabel {\small $\partial E_1$} [B] at 201 255
\pinlabel {\small $\partial E_2$} [B] at 392 255

\pinlabel {\small $\partial E_3$} [B] at 83 108
\pinlabel {\small $\partial E_4$} [B] at 274 108
\pinlabel {\small $\partial E_5$} [B] at 469 108
\pinlabel {\small $\partial E_3$} [B] at 5 42
\pinlabel {\small $\partial E_4$} [B] at 201 42
\pinlabel {\small $\partial E_5$} [B] at 392 42

\pinlabel {\small $\partial E_1$} [B] at 83 270
\pinlabel {\small $\partial E_2$} [B] at 274 270
\pinlabel {\small $\partial E_3$} [B] at 427 300
\pinlabel {\small $\partial E_4$} [B] at 55 143
\pinlabel {\small $\partial E_5$} [B] at 228 108

\pinlabel {\small $\alpha_0$} [B] at 110 341
\pinlabel {\small $\alpha_1$} [B] at 301 340

\pinlabel {\Large $\Sigma_0$} [B] at 83 215
\pinlabel {\Large $\Sigma_1$} [B] at 274 215
\pinlabel {\Large $\Sigma_2$} [B] at 469 215
\pinlabel {\Large $\Sigma_3$} [B] at 83 15
\pinlabel {\Large $\Sigma_4$} [B] at 274 15
\pinlabel {\Large $\Sigma_5$} [B] at 469 15

\pinlabel {\small $z$} [B] at 501 147
\pinlabel {\small $\partial E_4$} [B] at 485 58

\endlabellist
\centering
\includegraphics[width=0.9\textwidth]{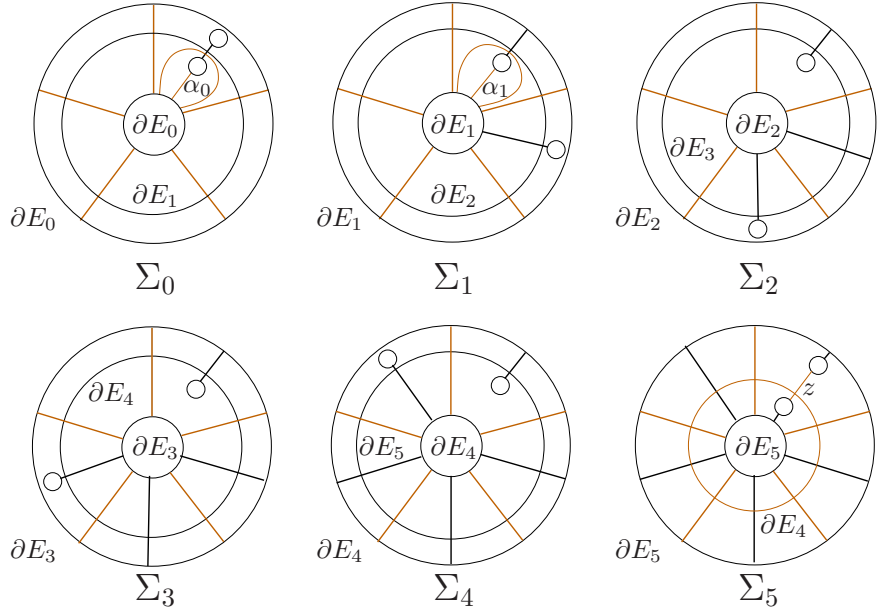}
\caption{The sequence of $2$-holed annuli from the consecutive surgeries for $L(5, 1)$.}
\label{sequence}
\end{figure}

Suppose that $C$ also intersects $E_1$, and let $C_1$ be an outermost subdisk of $C$ cut off by $C \cap E_1$.
Then $C_1 \cap \Sigma_1$ is the frontier of a regular neighborhood of the union of one hole of $\Sigma_1$ and an arc, say $\alpha_1$, connecting the hole to a boundary circle.
The arc $\alpha_1$ is also disjoint from $y \cap \Sigma_1$ by the same reason to $\alpha_0$.
Thus if we denoted by $E_2$ the disk from surgery on $E_1$ along $C_1$ that is not $E$, then $\partial E_2$ represents a word $xyxy^{p-1}$.
See $\Sigma_1$ Figure \ref{sequence}.

We continue such a construction repeatedly whenever $C$ also intersects the next disk.
For each $1 \leq j \leq p-1$, if $C$ intersects $E_j$, then we obtain the disk $E_{j+1}$ from surgery on $E_j$ along an outermost subdisk $C_j$ cut off by $C \cap E_j$.
We see that $|C \cap E_{j+1}| < |C \cap E_j|$ from the surgery construction.
In the $2$-holed annulus $\Sigma_j$, the arc $C_j \cap \Sigma_j$ is the frontier of a regular neighborhood of the union of a hole of $\Sigma_j$ and an arc $\alpha_j$ connecting the hole to a boundary circle.
The arc $\alpha_j$ is disjoint from $y \cap \Sigma_j$, and so $\partial E_{j+1}$ represents a word of the form $(xy)^jxy^{p-j}$.
In particular, notice that the disks $E_p$ is semiprimitive and $E_{p-1}$ is primitive, since there is a primitive disk $E''$ in $W$ disjoint from $\partial E_p$ and intersects $\partial E_{p-1}$ in a single point.
Such an $E''$ is not hard to find.
In the final $2$-holed annulus $\Sigma_5$ in Figure \ref{sequence}, the arc $z$ is the boundary circle of $E''$ in $\Sigma_p$.
Note that $z$ is disjoint from $x \cup y$, and so it does bound a disk $E''$ in the $3$-ball $W$ cut off by $E' \cup E'_0$.
Also $z$ intersects $\partial E_{p-1}$ in a single point and is disjoint from $\partial E_p$.

We remark that each of the arcs $\alpha_j$, $j \in \{0, 1, \cdots, p-1\}$, is disjoint from the circle $y$ due to the fact that $D$ is primitive. There are infinitely many arcs $\alpha_0$ that are not isotopic to each other in $\Sigma_0$, but each arc $\alpha_j$, $j \geq 1$, is unique up to isotopy in $\Sigma_j$.
Therefore, once $E_1$ is determined, we have the unique sequence of disks $E_2, E_3, \cdots, E_p$ only under the condition that each $\alpha_j$ is disjoint from $y$.

\medskip

{\noindent \sc Claim.}
For each $j \in \{2, 3, \cdots, p-1\}$, the subdisk $C$ intersects $E_j$.

\smallskip

{\noindent \it Proof of claim.}
Suppose not, and let $E_j$ be the first disk disjoint from $C$.
First, suppose that $j \in \{2, 3, \cdots, p-3\}$.
Then $C$ is disjoint from $E_j$ and intersects $E_{j-1}$, and so the arc $\partial C \cap \Sigma_j$ gives a subword of $\partial D$ of the form $(yx)^j y^{p-j}$ which implies that $D$ is not primitive by Lemma \ref{lem:criterion} again, a contradiction.
Next, suppose that $j = p-2$. That is, $C$ is disjoint from $E_{p-2}$ and intersects $E_{p-3}$.
Then one of the resulting disks from surgery on $E$ along $C$ is $E_{p-2}$, and the other one is exactly $E_{p-1}$ which is the disk in the sequence of disks in the previous construction.
The subdisk $C$ is disjoint from $E_{p-2} \cup E_{p-1}$, and consequently, $C$ intersects necessarily the semiprimitive disk $E_p$ in the previous construction in a single arc.
That is, $|C \cap E_p| = 1$.
But from the consecutive surgery constructions for $j \in \{2, 3, \cdots, p-3\}$, we have $1 \leq |C \cap E_{p-3}| <  |C \cap E_0|$, which contradicts the minimality of $|C \cap E_0|$.
Similarly, if $j = p-1$, then we have the same contradiction on the minimality, since $C$ is disjoint from $E_p$ in this case.

\medskip

By the claim, we can do surgery on $E_{p-1}$ along $C_{p-1}$ and one resulting disk from surgery is $E_p$, the semiprimitive disk.
But $|C \cap E_{j+1}| < |C \cap E_j|$ for each $j \in \{1, 2, \cdots, p-1\}$, and consequently $|C \cap E_p| < |C \cap E_0|$, which contradicts the minimality of $|C \cap E_0|$ again.

Therefore the primitive disk $E_1$ is a disk from surgery on $E$ along $C$, and $E'$ is also a dual disk of $E_1$, and so we complete the proof.
We note that the other disk from surgery is either $E_0$ or $E_2$ depending on the cases whether $C$ is disjoint from $E_0$ or not.
\end{proof}

\begin{theorem}
Let $(V, W; \Sigma)$ be the genus two Heegaard splitting of the lens space $L = L(p, 1)$, $p \geq 2$.
Then, for every primitive pair $\{D, E\}$ of $V$, $D$ and $E$ have a common dual disk.
In particular, the two disks of each primitive pair have a unique common dual disk if $p \geq 3$, and have exactly two common dual disks if $p = 2$ which form a primitive pair in $W$.
\label{thm:common_dual}
\end{theorem}

\begin{proof}
The proof of the existence of a common dual disk goes almost in the same line to that of Theorem \ref{thm:key}, by taking the primitive disk $D$ disjoint from $E$, in stead of the outermost sub-disk $C$ in Theorem \ref{thm:key}.
That is, when we choose a dual disk $E'$ of $E$ so that $|\partial D \cap \partial E_0|$ is minimal where $E_0$ is the unique semiprimitive disk in $V$ disjoint from $\Nbd(E \cup E')$, the primitive disk $D$ must be $E_1$, having the common dual disk $E'$ of $E$.

Now, let $E'$ be a common dual disk of $D$ and $E$.
Let $E_0$ and $E'_0$ be the unique meridian disks of $V$ and $W$ respectively that are disjoint from $\Nbd(E \cup E')$ (see Figure \ref{common_dual} (a)).
Cut the surface $\Sigma$ along $\partial E' \cup \partial E'_0$ to obtain the $4$-holed sphere $\Sigma'$.
Then $\partial E \cap \Sigma'$ is a single arc in $\Sigma'$ connecting the two holes coming from $\partial E'$, and $\partial D \cap \Sigma'$ consists of $p-1$ parallel arcs connecting the two holes coming from $\partial E'_0$ and the two arcs connecting $\partial E'$ to $\partial E'_0$ so that each of the endpoints of the two arcs meets exactly one hole of $\Sigma'$ (see Figure \ref{common_dual} (b)).

\begin{figure}
\labellist
\pinlabel {\small$E_0$} [B] at -5 65
\pinlabel {\small$D$} [B] at 93 52
\pinlabel {\small$E$} [B] at 175 65
\pinlabel {\small$\partial E'$} [B] at 117 85
\pinlabel {\small$\partial E'_0$} [B] at 55 25

\pinlabel {\small$\partial E'_+$} [B] at 373 97
\pinlabel {\small$\partial E'_0$} [B] at 218 97
\pinlabel {\small$\partial E'_-$} [B] at 373 40
\pinlabel {\small$\partial E'_0$} [B] at 218 40

\pinlabel {\small$\partial D$} [B] at 295 105
\pinlabel {\small$\partial D$} [B] at 234 69
\pinlabel {\small$\partial D$} [B] at 295 33
\pinlabel {\small$\partial D'$} [B] at 295 70
\pinlabel {\small$\partial E$} [B] at 356 69

\pinlabel {\small$\alpha'$} [B] at 329 89
\pinlabel {\small$\alpha''$} [B] at 331 51

\pinlabel {\Large$\Sigma$} [B] at 167 12
\pinlabel {\Large$\Sigma'$} [B] at 367 12

\pinlabel {\large(a)} [B] at 86  0
\pinlabel {\large(b)} [B] at 297 0

\endlabellist
\begin{center}
\includegraphics[width=0.85\textwidth]{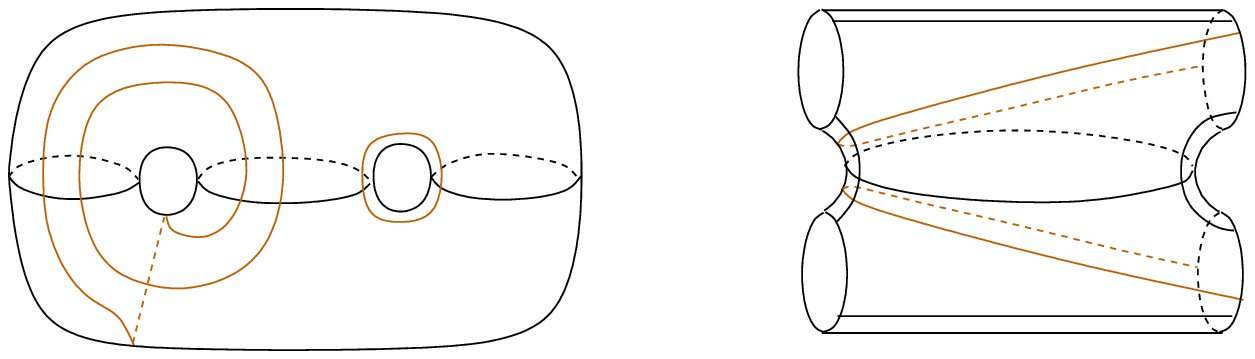}
\caption{The surfaces $\Sigma$ and $\Sigma'$ for $L(2, 1)$.}
\label{common_dual}
\end{center}
\end{figure}

Let $D'$ be a common dual disk of $D$ and $E$ which is not isotopic to $E'$.
Then an outermost sub-disk $C'$ of $D'$ cut off by $D' \cap (E' \cup E'_0)$ would intersects $\partial D$ if $C'$ is incident to $E'$.
Denote by $\partial E'_+$ and $\partial E'_-$ the two holes of $\Sigma'$ came from $\partial E'$.
We may assume that the endpoints of the arc $\alpha' = C' \cap \Sigma'$ meet $\partial E'_+$.
Since $|\partial D' \cap \partial E'_+| = |\partial D' \cap \partial E'_-|$, we must have one more arc component $\alpha''$ of $\partial D' \cap \Sigma'$ other than $C' \cap \Sigma'$ whose endpoints meet $\partial E'_-$ (see Figure \ref{common_dual} (b)).
The arc $\alpha''$ also intersects $\partial D$, and so $\partial D'$ intersects $\partial D$ in more than one point, which contradicts that $D'$ is a dual disk of $D$.
Similarly, if $C'$ is incident to $E'_0$, then $D'$ cannot be a dual disk of $E$.
Thus we see that $D'$ is disjoint from $E' \cup E'_0$.

If $p \geq 3$, there is no possibility of such a disk $D'$ which is disjoint from $E' \cup E'_0$ and is not isotopic to $E'$, and so $E'$ is the unique common dual disk.
If $p = 2$, then there is a unique circle in $\Sigma'$ which is not boundary parallel and which intersects each of $\partial E$ and $\partial D$ exactly once (see the circle $\partial D'$ in Figure \ref{common_dual} (b)).
So we have exactly two common dual disks $D'$ and $E'$ and in this case they are disjoint from each other.
\end{proof}

Given a primitive disk $D$ in $V$, there are infinitely many (non-isotopic) primitive disks each of which forms a primitive pair together with $D$.
But any primitive pair can be contained in at most one primitive triple, proved as follows.

\begin{theorem}
Let $(V, W; \Sigma)$ be the genus two Heegaard splitting of the lens space $L = L(p, 1)$, $p \geq 2$.
Then there is a primitive triple of $V$ if and only if $p = 3$.
In this case, every primitive pair is contained in a unique primitive triple.
\label{thm:dimension}
\end{theorem}

\begin{proof}
Let $\{E, E_1\}$ be a primitive pair of $V$.
Choose a common dual disk $E'$ of $E$ and $E_1$ given by Theorem \ref{thm:common_dual}.
There are unique semiprimitive disks $E_0$ in $V$ and $E'_0$ in $W$s disjoint from $\Nbd(E \cup E')$.
Let $\Sigma_1$ be the $4$-holed sphere $\partial V$ cut off by $\partial E \cup \partial E_1$, and as in Figure \ref{sequence} again, consider $\Sigma_1$ as a $2$-holed annulus with two boundary circles came from $\partial E_1$ and two holes from $\partial E$.
We give symbols $x$ and $y$ to $\partial E'$ and $\partial E'_0$ respectively as in the proof of Theorem \ref{thm:key}.

The boundary of any primitive disk $E_2$ in $V$ disjoint from $E$ and $E_1$, if exists, lies in $\Sigma_1$, and it is the frontier of a regular neighborhood of the union of a boundary circle, a hole of $\Sigma_1$ and an arc $\alpha_1$ connecting them.
This arc is disjoint from the arcs $y \cap \Sigma_1$, otherwise $\partial E_2$ represents a word containing $yxy^{-1}$, that is, $E_2$ is not primitive.
Consequently, $\partial E_2$ is uniquely determined and it represents a word of the form $xyxy^{p-1}$, and so it is primitive if and only if $p=3$.
Thus, only when $p = 3$, we have the unique primitive triple $\{E, E_1, E_2\}$ containing the pair $\{E, E_1\}$.
\end{proof}

\begin{figure}
\labellist
\pinlabel $E'$ [B] at 352 123
\pinlabel $D'$ [B] at 218 123
\pinlabel $F'$ [B] at 191 73
\pinlabel $\partial E$ [B] at 503 110
\pinlabel $\partial D$ [B] at 503 80
\pinlabel $\partial F$ [B] at 503 50
\pinlabel {\Large$W$} [B] at 187 4
\endlabellist
\begin{center}
\includegraphics[width=0.8\textwidth]{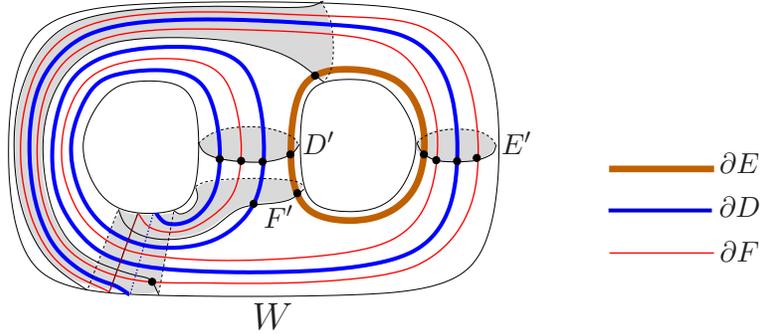}
\caption{The primitive triple $\{D', E', F'\}$ of $W$ in $L(3, 1)$ with the boundary circles $\partial D$, $\partial E$ and $\partial F$ of the disks in the primitive triple of $V$.}
\label{triple}
\end{center}
\end{figure}

\begin{remark}
For any primitive triple $\{D, E, F\}$ of $V$ in $L(3, 1)$, by Theorem \ref{common_dual}, there exist unique common dual disks $D'$, $E'$ and $F'$ of the disks in the pairs $\{E, F\}$, $\{F, D\}$ and $\{D, E\}$ respectively.
In fact, the disks $D'$, $E'$ and $F'$ form a primitive triple of $W$.
Furthermore, we have $|\partial D' \cap \partial D|=|\partial E' \cap \partial E|=|\partial F' \cap \partial F| = 2$.
Figure \ref{triple} illustrates the triple $\{D', E', F,\}$ of $W$ together with the boundary circles of $D$, $E$ and $F$ in $\partial W = \Sigma$.
\label{remark}
\end{remark}

\section{The complex of primitive disks}
\label{sec:primitive}

Let $M$ be an irreducible $3$-manifold with compressible boundary.
The {\it disk complex} of $M$ is a simplicial complex defined as follows.
The vertices of the disk complex are isotopy classes of essential disks in $M$, and a collection of $k+1$ vertices spans a $k$-simplex if and only if it admits a collection of representative disks which are pairwise disjoint.
In particular, if $M$ is a handlebody of genus $g \geq 2$, then the disk complex is $(3g - 4)$-dimensional and is not locally finite.
The following is a key property of a disk complex.

\begin{theorem}
If $\mathcal K$ is a full subcomplex of the disk complex satisfying the following condition, then $K$ is contractible.

\begin{itemize}
 \item Let $E$ and $D$ be disks in $M$ represent vertices of $\mathcal K$.
 If they intersect each other transversely and minimally, then at least one of the disks from surgery on $E$ along an outermost subdisk of $D$ cut off by $D \cap E$ represents a vertex of $\mathcal K$.
\end{itemize}
\label{thm:contractibility}
\end{theorem}

In \cite{C}, the above theorem is proved in the case that $M$ is a handlebody, but the proof is still valid for an arbitrary irreducible manifold with compressible boundary. From the theorem, we see that the disk complex itself is contractible.

Now consider the genus two Heegaard splitting $(V, W; \Sigma)$ of a lens space $L(p, 1)$, $p \geq 2$.
We define the {\it primitive disk complex}, denoted by $\mathcal P(V)$, to be the full subcomplex of the disk complex spanned by the vertices of primitive disks in $V$.
We already know that every primitive disk is a member of infinitely many primitive pair, and so every vertex of $\mathcal P(V)$ has infinite valency.
The following is our main theorem, which is a direct consequence of Theorems \ref{thm:key}, \ref{thm:dimension} and \ref{thm:contractibility}.

\begin{theorem}
Let $(V, W; \Sigma)$ be the genus two Heegaard splitting of the lens space $L = L(p, 1)$, $p \geq 2$.
The primitive disk complex $\mathcal P(V)$ is contractible.
In particular, if $p \neq 3$, it is a tree, and if $p = 3$, it is $2$-dimensional, and every edge is contained in a unique $2$-simplex.
\label{thm:main}
\end{theorem}

Figure \ref{complex} illustrates portions of the primitive disk complexes. The black vertices are the vertices of $\mathcal P(V)$ while the white ones the barycenters of edges when $p \neq 3$ and of $2$-simplices when $p = 3$.
Observe that the $2$-dimensional $\mathcal P(V)$ deformation retracts to a tree in its barycentric subdivision, as in the figure.

\begin{figure}
\labellist
\pinlabel (a) [B] at 116 0
\pinlabel (b) [B] at 401 0
\endlabellist
\begin{center}
\includegraphics[width=0.9\textwidth]{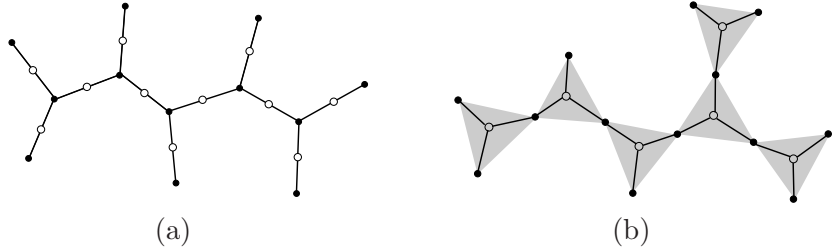}
\caption{Small portions of primitive disk complexes $\mathcal P(V)$ for $p \neq 3$ (a) and $p = 3$ (b).}
\label{complex}
\end{center}
\end{figure}

\section{Genus two Goeritz groups of lens spaces $L(p, 1)$}
\label{sec:presentations}
In this section, we give explicit presentation of the genus two Goeritz group $\mathcal G$ of each lens space $L(p, 1)$.
From Theorem \ref{thm:contractibility}, if  $p \neq 3$, the primitive disk complex $\mathcal P(V)$ is a tree, and if $p = 3$, then $\mathcal P(V)$ is $2$-dimensional but deformation retracts to a tree.
We simply denote by $\mathcal T$ the barycentric subdivision of the tree $\mathcal P(V)$ if $p \neq 3$ and of the deformation retract of $\mathcal P(V)$ if $p = 3$.
Each of the trees $\mathcal T$ is bipartite, as in Figure \ref{complex}, with the black vertices of (countably) infinite valence, and the white vertices of valence $2$ if $p \neq 3$ and of valence $3$ if $p = 3$.

Each black vertex of $\mathcal T$ is represented by a primitive disk, while each white vertex is represented by a primitive pair if $p \neq 3$, and by a primitive triple if $p = 3$.
An element of the group $\mathcal G$ can be considered as a simplicial automorphism of $\mathcal T$.
The tree $\mathcal T$ is invariant under the action of $\mathcal G$ for each $L(p, 1)$.
In particular, $\mathcal G$ acts transitively on each of the set of black vertices and the set of white vertices, and hence the quotient of $\mathcal T$ by the action of $\mathcal G$ is a single edge of which one end vertex is black and another one white.
Thus, by the theory of groups acting on trees due to Bass and Serre \cite{S}, the group $\mathcal G$ can be expressed as the free product of the stabilizer subgroups of two end vertices with amalgamated stabilizer subgroup of the edge.

First, we find a presentation of the stabilizer subgroup of a black vertex of $\mathcal T$, that is, of (the isotopy class of) a primitive disk in $V$.
For convenience, we will not distinguish disks (pairs and triples of disks) and homeomorphisms from their isotopy classes in their notations.
Throughout the section, $\mathcal G_{\{A_1, A_2, \cdots, A_k\}}$ will denote the subgroup of $\mathcal G$ of elements preserving each of $A_1, A_2, \cdots, A_k$ setwise, where $A_i$ will be (isotopy classes of) disks or union of disks in $V$ or in $W$.

\begin{lemma}
Let $E$ be a primitive disk in $V$.
The the stabilizer subgroup $\mathcal G_{\{E\}}$ of $E$ has the presentation
$\langle ~\alpha ~|~ \alpha^2 =1~\rangle \oplus \langle ~\beta, \gamma ~|~ \gamma^2 = 1 ~ \rangle,$
where the generators $\alpha, \beta$ and $\gamma$ are described in Figure \ref{disk_stabilizer}.
\label{lem:black_stabilizer}
\par
\end{lemma}
\begin{figure}
\labellist
\pinlabel (a) [B] at 100 -7
\pinlabel (b) [B] at 284 -7
\pinlabel {\small$\pi$} [B] at -5 74
\pinlabel {\small$\pi$} [B] at 177 74
\pinlabel {\small$\pi$} [B] at 386 77
\pinlabel {\small$\pi$} [B] at 327 -7
\pinlabel {\small$\beta'$} [B] at 2 60
\pinlabel {\small$\beta$} [B] at 167 60
\pinlabel {\small$\alpha$} [B] at 375 64
\pinlabel {\small$\gamma$} [B] at 340 5
\pinlabel {\small$E'_0$} [B] at 29 88
\pinlabel {\small$E'$} [B] at 137 88
\pinlabel {\small$E'_0$} [B] at 227 86
\pinlabel {\small$D'$} [B] at 306 84
\pinlabel {\small$E'$} [B] at 349 85
\pinlabel {\small$\partial E_0$} [B] at 63 37
\pinlabel {\small$\partial E$} [B] at 113 55
\pinlabel {\small$\partial E$} [B] at 337 57
\pinlabel {\large$W$} [B] at 165 27
\pinlabel {\large$W$} [B] at 357 27
\endlabellist
\begin{center}
\includegraphics[width=0.9\textwidth]{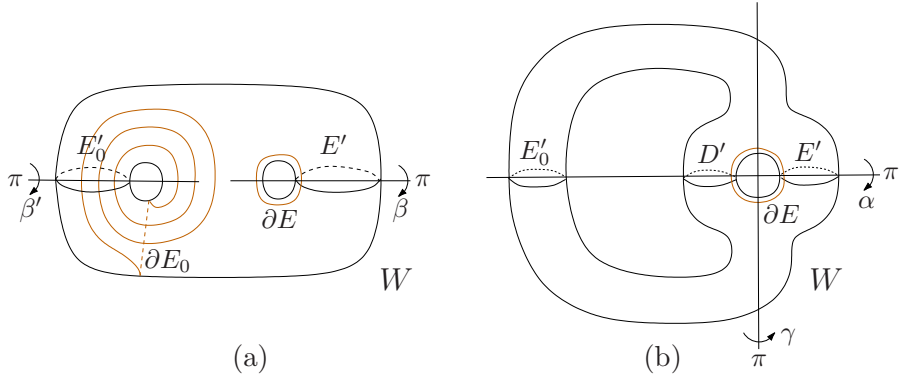}
\caption{Generators of the stabilizer subgroup $\mathcal G_{\{E\}}$.}
\label{disk_stabilizer}
\end{center}
\end{figure}
\begin{proof}
Let $\mathcal P'(W)$ be the full subcomplex of the primitive disk complex $\mathcal P(W)$ for $W$ spanned by the vertices of dual disks of $E$.
There is a unique semiprimitive disk $E'_0$ in $W$ disjoint from $\partial E$, and it is easy to show that any dual disk of $E$ is disjoint from $E'_0$.
Thus $\mathcal P'(W)$ is $1$-dimensional and further, by a similar argument to the case of $\mathcal P(V)$, we have that $\mathcal P'(W)$ is a tree whose vertices have infinite valence.
That is, when two dual disks of $E$ intersect each other, one of the two disks from surgery construction is $E'_0$ and the other one is again a dual disk of $E$.
Denote by $\mathcal T'$ the barycentric subdivision of $\mathcal P'(W)$.
The tree $\mathcal T'$ is invariant under the action of the stabilizer subgroup $\mathcal G_{\{E\}}$, and the quotient of $\mathcal T'$ by the action is a single edge.
One vertex of this edge corresponds to a dual disk $E'$ of $E$, and the other one to a primitive pair $\{E', D'\}$ of dual disks of $E$.
Thus $\mathcal G_{\{E\}}$ can be expressed as the free product of stabilizer subgroups $\mathcal G_{\{E, E'\}} \ast \mathcal G_{\{E, E' \cup D'\}}$ amalgamated by $\mathcal G_{\{E, E', D'\}}$.

Consider the subgroup $\mathcal G_{\{E, E'\}}$ first.
Any element of $\mathcal G_{\{E, E'\}}$ also preserves the disks $E_0$ and $E'_0$, which are unique meridian disks disjoint from $E \cup E'$ in $V$ and in $W$ respectively.
Since $V$ cut off by $E \cup E_0$ and $W$ cut off by $E' \cup E'_0$ are all $3$-balls, the group $\mathcal G_{\{E, E'\}}$ is identified with the group of isotopy classes of orientation preserving homeomorphisms of $\Sigma = \partial V = \partial W$ which preserve each of $\partial E$, $\partial E'$, $\partial E_0$ and $\partial E'_0$.
This group has a presentation $\langle ~\beta, \beta'~|~ (\beta \beta')^2 = 1, \beta\beta' = \beta'\beta ~\rangle$, where the generators $\beta$ and $\beta'$ are $\pi$-rotations (half Dehn twists) described in Figure \ref{disk_stabilizer} (a).

Next, consider the subgroup $\mathcal G_{\{E, E' \cup D'\}}$.
Any element of this group preserves $E' \cup D'$ in $W$, and further it preserves $E$ and $E_0 \cup D_0$ in $V$ where $E_0$ and $D_0$ are unique meridian disks in $V$ disjoint from $E \cup E'$ and $E \cup D'$ respectively.
Thus $\mathcal G_{\{E, E' \cup D'\}}$ is generated by two elements $\alpha$ and $\gamma$, where $\alpha$ is the hyperelliptic involution, and $\gamma$ is the element of order $2$ exchanging $E'$ and $D'$ described in Figure \ref{disk_stabilizer} (b).
Thus $\mathcal G_{\{E, E' \cup D'\}}$ has the presentation $\langle ~\alpha ~|~ \alpha^2 = 1 ~\rangle \oplus \langle~ \gamma ~|~ \gamma^2 = 1 ~\rangle$.
Similarly, $\mathcal G_{\{E, E', D'\}}$ has the presentation  $\langle ~\alpha ~|~ \alpha^2 = 1  ~\rangle$.
Observing that $\alpha$ satisfies  $\beta \beta' = \alpha$, we have the desired presentation of $\mathcal G_{\{E\}}$.
\end{proof}

Thus we have the same stabilizer subgroup of a black vertex for each $L(p, 1)$, $p \geq 2$, but for a white vertex, we have the following cases depending on $p$.

\begin{lemma}
A white vertex of $\mathcal T$ corresponds to a primitive pair if $p \neq 3$ and to a primitive triple if $p = 3$.
\begin{enumerate}
\item Let $\{D, E\}$ be a primitive pair of $V$ in $L(p, 1)$.
Then the stabilizer subgroup $\mathcal G_{\{D \cup E\}}$ has the presentation
$\langle ~\rho, \gamma ~|~ \rho^4 = \gamma^2 = (\rho \gamma )^2 = 1 ~ \rangle$ if $p=2$, and $\langle ~\alpha ~|~ \alpha^2 = 1 ~\rangle \oplus \langle ~\sigma ~|~ \sigma^2 = 1 ~ \rangle$ if $p\geq 3$, where the generators are described in Figure \ref{pair_stabilizer}.
\item Let $\{D, E, F\}$ be a primitive triple of $V$ in $L(3, 1)$.
Then the stabilizer subgroup $\mathcal G_{\{D \cup E \cup F\}}$ has the presentation
$\langle ~\alpha ~|~\alpha^2 = 1~ \rangle \oplus \langle ~ \delta, \gamma ~|~ \delta^3 = \gamma^2 = (\gamma \delta)^2 = 1 ~  \rangle$, where the generators are described in Figure \ref{triple_stabilizer}.
\par
\end{enumerate}
\label{lem:white_stabilizer}
\end{lemma}

\begin{figure}
\labellist
\pinlabel (a) [B] at 77 -9
\pinlabel (b) [B] at 262 -9
\pinlabel (c) [B] at 449 -9
\pinlabel {\small$\pi$} [B] at 178 93
\pinlabel {\small$\pi$} [B] at 119 12
\pinlabel {\small$\pi$} [B] at 263 12
\pinlabel {\small$\pi$} [B] at 486 12
\pinlabel {\small$\pi$} [B] at 545 93
\pinlabel {\scriptsize$\frac{\pi}{4}$} [B] at 322 141
\pinlabel {\scriptsize$\frac{\pi}{2}$} [B] at 304 35
\pinlabel {\small$\gamma$} [B] at 132 27
\pinlabel {\small$\alpha$} [B] at 165 80
\pinlabel {\small$\gamma$} [B] at 276 25
\pinlabel {\small$\rho$} [B] at 312 159
\pinlabel {\small$\alpha$} [B] at 323 41
\pinlabel {\small$\sigma$} [B] at 499 23
\pinlabel {\small$\alpha$} [B] at 532 81

\pinlabel {\scriptsize$D'$} [B] at 100 80
\pinlabel {\scriptsize$E'$} [B] at 140 80
\pinlabel {\scriptsize$E'$} [B] at 232 118
\pinlabel {\scriptsize$D'$} [B] at 232 54
\pinlabel {\scriptsize$D'$} [B] at 295 118
\pinlabel {\scriptsize$E'$} [B] at 295 54
\pinlabel {\scriptsize$D$} [B] at 465 80
\pinlabel {\scriptsize$E$} [B] at 509 80

\pinlabel {\scriptsize$\partial E$} [B] at 117 113
\pinlabel {\scriptsize$\partial D$} [B] at 25 75
\pinlabel {\scriptsize$\partial E$} [B] at 261 120
\pinlabel {\scriptsize$\partial E$} [B] at 261 50
\pinlabel {\scriptsize$\partial D$} [B] at 231 86
\pinlabel {\scriptsize$\partial D$} [B] at 295 86
\pinlabel {\scriptsize$\partial E'$} [B] at 486 112

\pinlabel {\large$W$} [B] at 155 48
\pinlabel {\large$V$} [B] at 520 48
\endlabellist
\begin{center}
\includegraphics[width=1.0\textwidth]{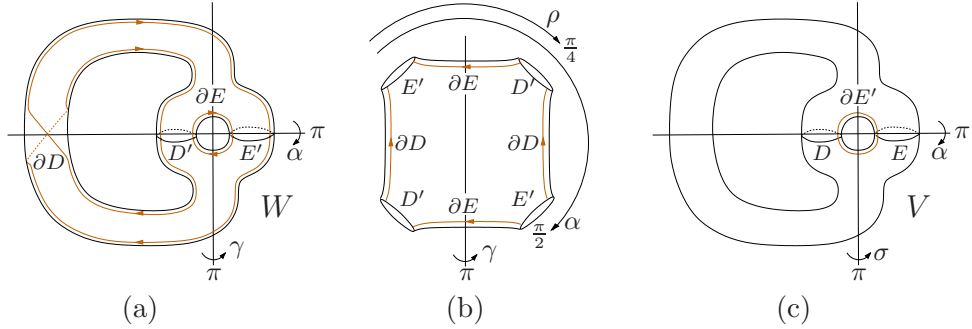}
\caption{Generators of the stabilizer subgroup $\mathcal G_{\{D \cup E\}}$ for $L(2, 1)$ in (a) and (b), and for $L(p, 1)$ with $p \geq 3$ in (c).}
\label{pair_stabilizer}
\end{center}
\end{figure}

\begin{proof}
\noindent(1)
First, let $\{D, E\}$ be a primitive pair of $V$ in $L(2, 1)$.
Then, by Theorem \ref{thm:common_dual}, there is a unique primitive pair $\{D', E'\}$ of $W$ such that each of $D'$ and $E'$ is a common dual disk of $D$ and $E$.
Any element of $\mathcal G_{\{D \cup E\}}$ preserves $D' \cup E'$, and hence $\mathcal G_{\{D \cup E\}}$ is identified with the stabilizer subgroup $\mathcal G_{\{D \cup E, D' \cup E'\}}$.
Since $D \cup E$ and $D' \cup E'$ cut off $V$ and $W$ into $3$-balls, the group $\mathcal G_{\{D \cup E, D' \cup E'\}}$ is identified with the group of isotopy classes of orientation preserving homeomorphisms of $\Sigma = \partial V = \partial W$ which preserve each of $\partial D \cup \partial E$ and $\partial D' \cup \partial E'$.
This is the dihedral group $D_8$ of order $8$ with generators $\rho$ and $\gamma$, described in Figure \ref{pair_stabilizer} (a) and (b).
The $3$-ball in Figure \ref{pair_stabilizer} (b) is $W$ cut off by $D' \cup E'$.
Figures \ref{pair_stabilizer} (a) and \ref{pair_stabilizer} (b) give two descriptions of each of the elements $\alpha$ and $\gamma$.
Thus we have the presentation $\langle ~\rho, \gamma ~|~ \rho^4 = \gamma^2 = (\rho \gamma )^2 = 1 ~ \rangle$.
We remark that the hyperelliptic involution $\alpha$ equals $\rho^2$.

Next, let $\{D, E\}$ be a primitive pair of $V$ in $L(p, 1)$ with $p \geq 3$.
There is a unique common dual disk $E'$ of $D$ and $E$ by Theorem \ref{thm:common_dual}, and hence $\mathcal G_{\{D \cup E\}}$ is identified with the stabilizer subgroup $\mathcal G_{\{D \cup E, E'\}}$.
As in the case of $\mathcal G_{\{E, E' \cup  D'\}}$ in the proof of Lemma \ref{lem:black_stabilizer}, this group is generated by two elements.
One is the hyperelliptic involution $\alpha$, and the other one is the element, denoted by $\sigma$, of order $2$ exchanging $D$ and $E$ described in Figure \ref{pair_stabilizer} (c).
Thus we have the presentation $\langle ~\alpha ~ |~ \alpha^2 = 1 ~\rangle \oplus \langle ~ \sigma ~|~  \sigma^2 = 1 ~ \rangle$.

\begin{figure}
\labellist
\pinlabel {\small $E'$} [B] at 195 108
\pinlabel {\small $D'$} [B] at 133 133
\pinlabel {\small $F'$} [B] at 178 91
\pinlabel {\small $\partial E$} [B] at 526 110
\pinlabel {\small $\partial D$} [B] at 526 91
\pinlabel {\small $\partial F$} [B] at 526 72

\pinlabel {\small $E'$} [B] at 402 79
\pinlabel {\small $D'$} [B] at 277 78
\pinlabel {\small $F'$} [B] at 344 183

\pinlabel $\gamma$ [B] at 175 35
\pinlabel $\gamma'$ [B] at 360 50
\pinlabel $\alpha$ [B] at 222 109
\pinlabel $\delta$ [B] at 420 165
\pinlabel $\pi$ [B] at 158 17
\pinlabel $\pi$ [B] at 235 125
\pinlabel $\pi$ [B] at 341 32
\pinlabel $\frac{\pi}{3}$ [B] at 420 70

\pinlabel {\Large $W$} [B] at 16 42
\pinlabel {\Large $B$} [B] at 388 44

\pinlabel (a) [B] at 110 2
\pinlabel (b) [B] at 345 2
\endlabellist
\begin{center}
\includegraphics[width=1.05\textwidth]{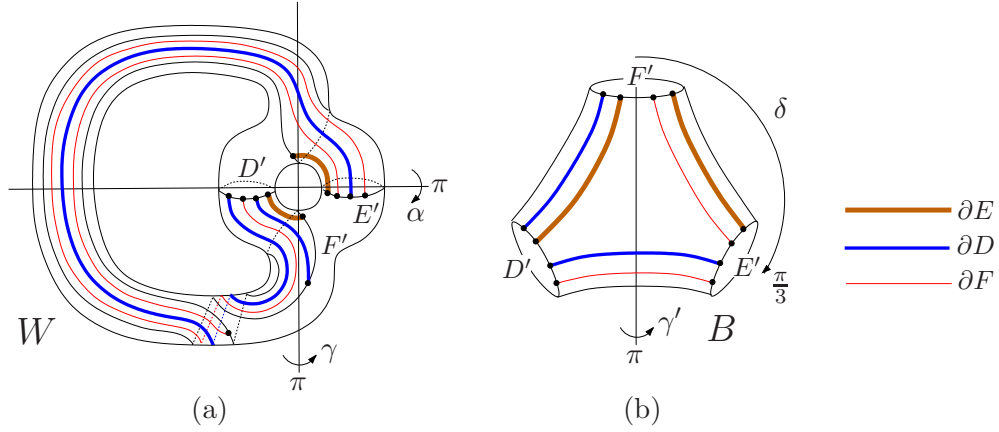}
\caption{(a) The primitive triple $\{D', E', F'\}$ of $W$, and the arcs $(\partial D \cup \partial E \cup \partial F) \cap \partial B$. (b) The $3$-ball $B$.}
\label{triple_stabilizer}
\end{center}
\end{figure}

\noindent(2) Let $\{D, E, F\}$ be a primitive triple of $V$ in $L(3, 1)$.
Then there exists a unique primitive triple $\{D', E', F'\}$ of $W$ as described in Remark \ref{remark} with Figure \ref{triple}.
Thus the stabilizer subgroup $\mathcal G_{\{D \cup E \cup F\}}$ is identified with $\mathcal G_{\{D \cup E \cup F, D' \cup E' \cup F'\}}$.
The union of three disks $D' \cup E' \cup F'$ cuts off $W$ into two $3$-balls.
One of them, say $B$, is shown in Figure \ref{triple_stabilizer} (b).
Consider the group of isotopy classes of orientation preserving homeomorphisms of $B$ which preserve each of $D' \cup E' \cup F'$ and $(\partial D \cup \partial E \cup \partial F) \cap \partial B$ on the boundary.
This group is the dihedral group $D_6 = \langle ~ \delta, \gamma' ~|~ \delta^3 = \gamma'^2 = (\gamma' \delta)^2 = 1 ~  \rangle$ of order $6$ with generators $\delta$ and $\gamma'$ in Figure \ref{triple_stabilizer} (b).
The element $\gamma$ in Figure \ref{triple_stabilizer} (a) is different from $\gamma'$, since $\gamma$ exchanges the two $3$-balls.
But they are related by $\gamma = \alpha \gamma'$, where $\alpha$ is the hyperelliptic involution exchanging the two $3$-balls as described in Figure \ref{triple_stabilizer} (a).
Thus we see that the relation $(\gamma' \delta)^2 = 1$ in $D_6$ is equivalent to $(\gamma \delta)^2 = 1$.
Since the elements $\alpha$, $\gamma$ and $\delta$ extend to elements of $\mathcal G_{\{D \cup E \cup F, D' \cup E' \cup F'\}}$, this group can be considered as the extension of $D_6$ by $\langle ~\alpha ~|~ \alpha^2 =1 ~\rangle$ with relations $\alpha \gamma \alpha = \gamma$ and $\alpha \delta \alpha = \delta$.
Thus we have the desired presentation of $\mathcal G_{\{D \cup E \cup F\}}$.
\end{proof}

Finally, the stabilizer subgroups of an edge are calculated in a similar way.

\begin{lemma}
An edge of $\mathcal T$ corresponds to the pair of end vertices.
\begin{enumerate}
\item Let $\{D, E\}$ be a primitive pair of $V$ in $L(p, 1)$.
Then $\mathcal G_{\{E, D\cup E\}} = \mathcal G_{\{E, D\}}$ has a presentation $\langle ~\alpha ~| ~\alpha^2 = 1~\rangle \oplus \langle ~\gamma ~| ~\gamma^2 = 1~\rangle$ if $p = 2$, and a presentation $\langle ~\alpha ~| ~\alpha^2 = 1~\rangle$ if $p \geq 3$.
\item Let $\{D, E, F\}$ be a primitive triple of $V$ in $L(3, 1)$.
Then $\mathcal G_{\{E, D \cup E \cup F\}} = \mathcal G_{\{E, D \cup F\}}$ has a presentation $\langle ~\alpha ~| ~\alpha^2 = 1~\rangle \oplus \langle ~\gamma ~| ~\gamma^2 = 1~\rangle$.
\end{enumerate}
\label{lem:edge_stabilizer}
\end{lemma}

Combining Lemmas \ref{lem:black_stabilizer}, \ref{lem:white_stabilizer} and \ref{lem:edge_stabilizer}, we obtain the main result.

\begin{theorem}
The genus-$2$ Goeritz group $\mathcal G$ of a lens space $L(p, 1)$, $p \geq 2$, has the following presentations.
\begin{enumerate}
\item $\langle ~\beta, \rho, \gamma~ |~ \rho^4 = \gamma^2 = (\gamma \rho)^2 = \rho^2 \beta \rho^2 \beta^{-1} = 1 ~ \rangle$ if $p = 2$,
\item $\langle ~\alpha ~ | ~ \alpha^2 = 1 ~ \rangle \oplus \langle ~\beta, \delta, \gamma~ |~ \delta^3 = \gamma^2 = (\gamma \delta)^2 = 1 ~ \rangle$ if $p = 3$, and
\item $\langle ~\alpha ~ | ~ \alpha^2 = 1 ~ \rangle \oplus \langle ~\beta, \gamma, \sigma ~ |~ \gamma^2 = \sigma^2 = 1 ~ \rangle$ if $p \geq 4$.
\end{enumerate}
\label{thm:presentation}
\end{theorem}

\smallskip
\noindent {\bf Acknowledgments.}
The author wish to express his gratitude to Darryl McCullough and Yuya Koda for helpful discussions for their valuable advice and comments.

\bibliographystyle{amsplain}

\begin{thebibliography}{20}

\bibitem{Ak}E. Akbas,
\textit{A presentation for the automorphisms of the 3-sphere that
preserve a genus two Heegaard splitting},
Pacific J. Math. \textbf{236} (2008), no. 2, 201--222.


\bibitem{B}F. Bonahon,
\textit{Diff\'eotopies des espaces lenticulaires},
Topology \textbf{22} (1983), no. 3, 305--314.

\bibitem{B-O}F. Bonahon, J.-P. Otal,
\textit{Scindements de Heegaard des espaces lenticulaires},
Ann. Sci. \'Ecole Norm. Sup. (4) \textbf{16} (1983), 451--466.


\bibitem{C}S. Cho,
\textit{Homeomorphisms of the 3-sphere that preserve a Heegaard splitting of genus two},
Proc. Amer. Math. Soc. \textbf{136} (2008), no. 3, 1113--1123.


\bibitem{C-K}S. Cho, Y. Koda,
\textit{Primitive disk complexes for lens spaces},
preprint.


\bibitem{Ge}L. Goeritz,
\textit{Die Abbildungen der Berzelfl\"{a}che und der Volbrezel vom Gesschlect 2},
Abh. Math. Sem. Univ. Hamburg \textbf{9} (1933), 244--259.

\bibitem{Go}C. McA. Gordon, \textit{On primitive sets of loops
in the boundary of a handlebody}, Topology Appl. \textbf{27}
(1987), no. 3, 285--299.

\bibitem{Kod}Y. Koda,
\textit{Automorphisms of the 3-sphere that preserve spatial graphs
and handlebody-knots}, arXiv:1106.4777.

\bibitem{McC}D. McCullough, \textit{Virtually geometrically finite
mapping class groups of 3-manifolds}, J. Differential Geom.
\textbf{33}, (1991), no. 1, 1--65.

\bibitem{O-Z}R.P. Osborne, H. Zieschang,
\textit{Primitives in the free group on two generators},
Invent. Math. \textbf{63} (1981), no. 1, 17--24.

\bibitem{R}D. Rolfsen,
\textit{Knots and links},
Mathematics Lecture Series, No. 7. Publish or Perish, Inc., 1976. ix+439 pp


\bibitem {Sc}M. Scharlemann,
\textit{Automorphisms of the 3-sphere that preserve a genus two Heegaard splitting},
Bol. Soc. Mat. Mexicana (3) \textbf{10} (2004), Special Issue, 503--514.

\bibitem{S}J. Serre,
\textit{Trees},
Springer-Verlag, 1980. ix+142 pp.


\bibitem{W}F. Waldhausen,
\textit{Heegaard-Zerlegungen der $3$-Sph\"{a}re},
Topology \textbf{7} (1968),  195--203.



\end{thebibliography}

\end{document}